\documentclass[11pt,reqno]{amsart}

\usepackage[a4paper,margin=1in]{geometry}

\usepackage{amsmath,amssymb,amsfonts,amsthm,mathtools}
\usepackage{graphicx}
\usepackage{booktabs}
\usepackage{threeparttable}
\usepackage{multicol}
\usepackage{microtype}
\usepackage{hyperref}

\numberwithin{equation}{section}

\theoremstyle{plain}
\newtheorem{thm}{Theorem}[section]
\newtheorem{theorem}[thm]{Theorem}
\newtheorem{lemma}[thm]{Lemma}

\theoremstyle{definition}

\theoremstyle{remark}

\title[Jacobsthal--Fibonacci Numbers]
{On a Diophantine Equation with Jacobsthal and Fibonacci Numbers}

\author{Daeyeoul Kim}
\address{Department of Mathematics,
Jeonbuk National University,
Jeonju, Republic of Korea}
\email{kdaeyeoul@jbnu.ac.kr}

\author{Zekiye Pinar Cihan}
\address{Department of Mathematics,
Bilecik Seyh Edebali University,
Bilecik, Turkey}
\email{pinarcihan@icloud.com}

\author{Zeynep Demirkol Ozkaya}
\address{Department of Medical Services and Technique,
Muradiye Vocational School,
Van Yuzuncu Yil University,
Van, Turkey}
\email{zeynepdemirkolozkaya@yyu.edu.tr}

\author{Ilker Inam}
\address{Department of Mathematics,
Bilecik Seyh Edebali University,
Bilecik, Turkey}
\email{ilker.inam@gmail.com}
\email{ilker.inam@bilecik.edu.tr}



\begin{document}

\begin{abstract}
In the present paper, we identify all Jacobsthal numbers that may be expressed as a product of three Fibonacci numbers. More precisely, our main result shows that the only solution to the Diophantine equation
\[
F_kF_lF_m=J_n
\]
for $2<k<l<m$ is
\[
(k,l,m,n)=(5,7,8,12).
\]
The proof relies on techniques involving linear forms in logarithms.
\end{abstract}

\maketitle

\section{Introduction}In a recent study, \cite{10}, the authors identified all Tribonacci numbers that can be expressed as the product of two Fibonacci numbers. In fact, there are only five nontrivial solutions. Inspired by this work, we obtain a new result that is similarly interesting to as the one in \cite{10}, but in a slightly modified equation. 
In \cite{5}, \cite{6}, \cite{8}, \cite{15}, and in a more general case in \cite{14}, similar problems have been studied for Fibonacci and Pell numbers. This naturally leads to the question of what occurs in the case of Fibonacci and Jacobsthal numbers.In a more recent paper, \cite{9}, the authors determine all Fibonacci numbers that result from the multiplication of two Jacobsthal numbers.This paper focuses on the case of Jacobsthal numbers and their representation as the product of three Fibonacci numbers. This equation is significant because it explores the interplay between two well-known recursive sequences, the Fibonacci and Jacobsthal numbers, in a multiplicative setting. Understanding multiplicative relations between distinct recursive sequences provides insight into the arithmetic rigidity and coincidence structure of such sequences. Note that it is also possible to solve certain Diophantine equations not only for products of specific number sequences but also for sums or differences, as demonstrated in \cite{3}, \cite{2}, \cite{16} and \cite{11}.We recall that the \textit{Fibonacci sequence} $\{F_n\}_{n\geq 0}$ is defined by $F_0 =0, F_1 = 1$ and the recurrence relation $F_{n+2} = F_n + F_{n+1}$, for all $n\geq 0$. Similarly, \textit{the Jacobsthal sequence} $\{J_n\}_{n\geq 0}$ is defined by $J_0 = 0$, $J_1 = 1$, and the recurrence relation $J_{n+2} = 2J_n + J_{n+1}$ for all $n\geq 0$. The Binet formulas of these sequences are:\begin{equation}\label{e:f}	F_n := \frac{\alpha^n-\beta^n}{\alpha-\beta}, \quad \text{for all} \quad n\geq 0,\end{equation}where $\alpha = \frac{1+\sqrt{5}}{2}$ and $\beta = \frac{1-\sqrt{5}}{2}$, and\begin{equation}\label{e:j}	J_n := \frac{2^n - (-1)^n}{3} \quad \text{for all} \quad n\geq 0,\end{equation}respectively.

\section{Main Theorem}In this paper, we investigate the problem of identifying all Jacobsthal numbers that can be expressed as the product of three Fibonacci numbers. Specifically, we provide a complete solution to the Diophantine equation\begin{equation}\label{e:d1}	F_kF_lF_m = J_n \end{equation}for all positive integers $(k, l, m, n)$. The following constitutes the primary contribution of the paper.\begin{theorem}\label{t:main}	There is exactly one non-trivial solution $(k, l, m, n)$ to the Diophantine equation above, which is	\begin{equation*}		\{ (5, 7, 8, 12)\}, \quad \text{for} \quad 2<k<l<m.	\end{equation*}	and this solution corresponds to	\begin{equation*}		F_5F_7F_8 = 5\cdot 13\cdot 21 = 1365 = J_{12}.	\end{equation*}\end{theorem}
\section{Preliminaries}For more details about this section, the reader refers to \cite{4}.Let $\alpha$ be an algebraic number of degree $d$ whose minimal polynomial is defined over $\mathbb{Z}$\begin{equation*}	c_0x^d + c_1x^{d-1} +\cdots + c_d = c_0\prod_{i=1}^{d} (x-\alpha^{(i)}) ,\end{equation*}where the leading coefficient $c_0$ is positive and the $\alpha^{(i)}$'s are conjugates of $\alpha$. Then the logarithmic height of $\alpha$ is given by \begin{equation}\label{e:h}	h(\alpha) := \frac{1}{d}(\log c_0 + \sum_{i=1}^{d}\log(\max\{\alpha^{(i)}, 1\})).\end{equation}In particular, if $\alpha = p/q$ is rational number with $\gcd(p, q)=1$, and $q \neq 0$, then \begin{equation*}	h(\alpha)=\log\max\{|p|, q\}.\end{equation*}The following are some of the properties of the logarithmic height function $h(\cdot)$, which will be used in the next section of this paper:\begin{align}\label{h:p}	h(\alpha_1\mp \alpha_2) &\leq h(\alpha_1) + h(\alpha_2) + \log 2;\\	h(\alpha_1\alpha_2^{\mp 1})& \leq h(\alpha_1) + h(\alpha_2);\label{h:p1}\\	h(\alpha^s)&=|s|h(\alpha), \qquad (s\in \mathbb{Z}).\label{h:p2}\end{align}We will also use the following important result due to Matveev.\begin{theorem}\cite{13} \label{t:M}	Let $\alpha_1, \alpha_2, \dots, \alpha_t$ be positive real algebraic numbers in a real algebraic number field $\mathbb{K}$ of degree $D$, $b_1,\dots, b_t$ be non-zero integers, and assume that 	\begin{equation}		\Lambda := \alpha_1^{b_1}\dots \alpha_t^{b_t},	\end{equation}	is non-zero. Then 	\begin{equation}		\log |\Lambda| > -1.4\times 30^{t+3}\times t^{4.5}\times D^2(1+\log D)(1+\log B)A_1\dots A_t,	\end{equation}	where	\begin{equation*}		B\geq \max\{|b_1|,\dots,|b_t|\},	\end{equation*}	and	\begin{equation*}		A_i\geq \max\{Dh(\alpha_i), |\log(\alpha_i)|, 0.16\}, \qquad \text{for all}\qquad i=1,\dots, t.	\end{equation*}\end{theorem} For the last result of preliminaries section, let $X$ be a real number. Set $||X|| := \min\{|X-n| : n\in\mathbb{Z}\}$ for the distance from $X$ to the nearest integer.\begin{lemma} \label{Lem} \cite{12}If $l \geqslant 1$, $H> (4l^{2})^{l}$ and $H> \frac{L}{(\log L )^{l}}$ then $$L < 2^{l} H ( \log H)^{l}.$$\end{lemma}\begin{lemma} \cite{1}, \cite{7} \label{l:1}	Let $M$ be a positive integer, $p/q$ be a convergent of the continued fraction of the irrational number $\tau$ such that $q>6M$, and $A, B, \mu$ be some real numbers with $A>0$ and $B>1$. Let further $\epsilon := ||\mu q|| - M||\tau q||$. If $\epsilon > 0$, then there is no solution to the inequality	\begin{equation*}		0<|u\tau - v + \mu|<AB^{-w}	\end{equation*}	in positive integers $u$, $v$ and $ w$ with	\begin{equation*}		u\leq M \qquad \text{and}\qquad w\geq \frac{\log (Aq/\epsilon)}{\log B}.	\end{equation*}\end{lemma}\section{Proof of The Main Theorem}	Note that, one can reformulate the equation (\ref{e:d1}) by (\ref{e:f}) and (\ref{e:j}) 		\begin{equation}\label{e:d2}		F_kF_lF_m = \frac{1}{5\sqrt{5}}(\alpha^{k}-\beta^{k}) (\alpha^l-\beta^{l})(\alpha^m- \beta^{m}) = \frac{2^n-(-1)^n}{3},	\end{equation}	for some positive integers $k,l, m, n$.		We re-arrange the equation (\ref{e:d2}), then we get			\begin{align*}		F_k F_l F_m &= \frac{1}{5\sqrt{5}} \Big( \alpha^{k+l+m} - \beta^{k+l+m} - \alpha^{k+l}\beta^m - \alpha^{k+m}\beta^l + \alpha^k \beta^{l+m} - \alpha^{l+m} \beta^k+ \alpha^l \beta^{k+m} + \alpha^m \beta^{k+l} \Big) \\&= \frac{1}{5\sqrt{5}} \Big( \alpha^{k+l+m} - (-1)^{k+l+m} \alpha^{-(k+l+m)} - \alpha^{k+l} (-1)^m \alpha^{-m} - \alpha^{k+m} (-1)^l \alpha^{-l}+ \alpha^k (-1)^{l+m} \alpha^{-(l+m)} \\&\quad - \alpha^{l+m} (-1)^k \alpha^{-k} + \alpha^l (-1)^{k+m} \alpha^{-(k+m)} + \alpha^m (-1)^{k+l} \alpha^{-(k+l)} \Big) \\&\quad=\frac{2^n-(-1)^n}{3} = J_n,	\end{align*}since $\alpha\beta = -1$ , i.e. $\beta = -1/\alpha$, $\alpha/\beta = -\beta^{-2}$, and $\beta/\alpha = -\alpha^{-2}$. Thus, we have	\begin{equation} \label{zzz}\begin{aligned} \frac{1}{5\sqrt{5}}(\alpha^{k+l+m})-\frac{2^n}{3} 	&\quad = -\frac{(-1)^{n}}{3}+\frac{1}{5\sqrt{5}} \Big(- (-1)^{k+l+m} \alpha^{-(k+l+m)} - \alpha^{k+l} (-1)^m \alpha^{-m} \\&\quad- \alpha^{k+m} (-1)^l \alpha^{-l} + \alpha^k (-1)^{l+m} \alpha^{-(l+m)} - \alpha^{l+m} (-1)^k \alpha^{-k}\\ &\quad + \alpha^l (-1)^{k+m} \alpha^{-(k+m)}+ \alpha^m (-1)^{k+l} \alpha^{-(k+l)} \Big).\\\end{aligned}\end{equation}	First, it is well-established that the characteristic polynomial for Fibonacci numbers is	\begin{equation*}		\Psi(X) = X^2 - X + 1 = 0	\end{equation*}	and the roots of $\Psi(X)$ are $\{\alpha, \beta\} = \{\frac{1+\sqrt{5}}{2}, \frac{1-\sqrt{5}}{2}\}$. Let $\mathbb{K}:= \mathbb{Q}(\alpha, \beta) = \mathbb{Q}(\sqrt{5})$ be the splitting field of the polynomial $\Psi$ over $\mathbb{Q}$. Then, it is clear that $[\mathbb{K} : \mathbb{Q}] = 2$.\\	Secondly, we know that the characteristic polynomial of Jacobsthal numbers is 	\begin{equation*}		P(X) = X^2 - X - 2 ,	\end{equation*}	its roots lie in $\mathbb{Q}$.	By induction, one can easily find	\begin{equation}\label{b:j1}		\alpha^{n-2}\leq F_n\leq \alpha^{n-1}, \quad \text{for}\quad n\geq 1,	\end{equation}	and	\begin{equation}\label{ilk1}		2^{n-2}\leq J_n\leq 2^{n-1}, \quad \text{for}\quad n\geq 1.	\end{equation}	After a short calculation, we find the following:	\begin{equation}\label{ilk2}		\alpha^{k+l+m-6}\leq F_kF_lF_m =J_n \leq 2^{n-1}	\end{equation}and	\begin{equation}\label{i:ff}		2^{n-2}\leq J_n = F_kF_lF_m \leq \alpha^{k+m+l-3}.	\end{equation}	The following lemma provides an expression for $n$ in terms of $k,l$ and $m$, which we now prove.\begin{lemma}\label{l:m}		All solutions of the Diophantine equation (\ref{e:d1}) satisfy		\begin{equation} \label{1}			(k+l+m) \frac{\log \alpha}{\log 2} -3.17\leq n\leq (k+l+m)\frac{\log \alpha}{\log2}-0.08.		\end{equation}\end{lemma}\begin{proof}		The proof follows easily from the inequality (\ref{ilk2}). One can see that		\begin{equation*}			\alpha^{k+l+m-6}\leq F_kF_lF_m = J_n\leq 2^{n-1}.		\end{equation*}Taking logarithms on both sides, we obtain an inequality of the form		\begin{equation*}			(k+l+m-6)\log \alpha \leq (n-1)\log 2,		\end{equation*}and hence		\begin{equation}\label{i:first}			 (k+l+m-6) \frac{\log \alpha}{\log 2} +1\leq n.		\end{equation}		Now, from (\ref{i:ff}) we rewrite that		\begin{equation*}			2^{n-2}\leq J_n\leq \alpha^{k+l+m-3}.		\end{equation*}		Taking the logarithm on both sides, then we get that		\begin{equation*}			(n-2)\log 2\leq (k+l+m-3)\log \alpha,		\end{equation*}		and then 		\begin{equation}\label{ln:i}			n\leq (k+l+m-3)\frac{\log \alpha}{\log2}+2.		\end{equation}Thus, from the inequalities (\ref{i:first}) and (\ref{ln:i}) we deduce that		\begin{equation}			(k+l+m-6) \frac{\log \alpha}{\log 2} +1\leq n\leq (k+l+m-3)\frac{\log \alpha}{\log2}+2.		\end{equation}		So, we obtain that 		\begin{equation}			(k+l+m) \frac{\log \alpha}{\log 2} -3.17\leq n\leq (k+l+m)\frac{\log \alpha}{\log2}-0.08.		\end{equation}		This concludes the proof.	\end{proof}			Now, in inequality (\ref{1}), it is written $n\leq (k+l+m)\frac{\log \alpha}{\log2}-0.08.$ Note that $\frac{\log \alpha}{\log2}\cong 0.69$ and $k < l< m.$ Then we have $$n\leq (k+l+m)\frac{\log \alpha}{\log2}-0.08 < (k+l+m) \cdot 1 -0.08 < (m+m+m)\cdot 1 -0.08 < 3m$$Thus we get that $n < 3m$. 		
In (\ref{zzz}), we take absolute value and divide by $\frac{1}{5\sqrt{5}} \alpha^{k+l+m} $ both sides of the rearranged equation above, then we get 		{\small\begin{align*}\left| \frac{2^{n}\cdot 5\sqrt{5}}{3\alpha^{k+l+m}} - 1 \right| &= \Biggl| - \frac{(-1)^{n} 5\sqrt{5}}{3\alpha^{k+l+m}} - \frac{(-1)^{m}}{\alpha^{2m}} - \frac{(-1)^{l}}{\alpha^{2l}} + \frac{(-1)^{l+m}}{\alpha^{2(l+m)}} - \frac{(-1)^{k}}{\alpha^{2k}} \\[-2pt]&\quad + \frac{(-1)^{k+m}}{\alpha^{2(k+m)}} + \frac{(-1)^{k+l}}{\alpha^{2(k+l)}} - \frac{(-1)^{k+l+m}}{\alpha^{2(k+l+m)}}\Biggr| \\[4pt]&\le \frac{5\sqrt{5}}{3\alpha^{k+l+m}}+ \alpha^{-2m}+ \alpha^{-2l}+ \alpha^{-2k}+ \alpha^{-2(l+m)} \\[-2pt]&\quad+ \alpha^{-2(k+m)}+ \alpha^{-2(k+l)}+ \alpha^{-2(k+l+m)} \\[4pt]&\le \alpha^{-2k}\!\left(\frac{5\sqrt{5}}{3}+7\right).\end{align*}}Thus, we have 	\begin{equation}\label{i:lam}		\left| \frac{2^{n}\cdot 5\sqrt{5}}{3\alpha^{k+l+m}}-1 \right| \leq 10.73 \cdot \alpha^{-2k}.	\end{equation}		Put		\begin{equation}\label{e:lam}		\Lambda_1 := 	 \frac{2^{n}\cdot 5\sqrt{5}}{3\alpha^{k+l+m}}-1.	\end{equation}		We claim that $\Lambda_1\neq 0$. To see this, we consider the $\mathbb{Q}$- automorphism $\sigma$ of the Galois extension $\mathbb{K} = \mathbb{Q}(\alpha, \beta) = \mathbb{Q}(\sqrt{5})$ over $\mathbb{Q}$ defined by $\sigma(\alpha) = \alpha$, $\sigma(\beta) = \beta$. Thus, $D = [\mathbb{Q}(\sqrt{5}): \mathbb{Q}] = 2$. If $\Lambda_1 = 0$, then we obtain\begin{equation*}5\sqrt{5} = 2^{-n}\cdot 3\alpha^{k+l+m}.\end{equation*}Multiplying both sides by $2^n$, we get\begin{equation*}5\cdot 2^n \sqrt{5} = 3\alpha^{k+l+m}.\end{equation*}Let $t = k+l+m$. Since $\alpha = (1+\sqrt{5})/2$, it follows that\[\alpha^t = \frac{A_1 + B_1 \sqrt{5}}{2^t}\]for some integers $A_1$ and $B_1$. Hence,\[3\alpha^t = \frac{3A_1}{2^t} + \frac{3B_1}{2^t}\sqrt{5}.\]Substituting this into the above equality and multiplying both sides by $2^t$, we obtain\[5\cdot 2^{n+t}\sqrt{5} = 3A_1 + 3B_1 \sqrt{5}.\]Comparing the rational and irrational parts in $\mathbb{Q}(\sqrt{5})$, we get\[3A_1 = 0 \quad \text{and} \quad 3B_1 = 5\cdot 2^{n+t}.\]Thus $A_1=0$ and $B_1 = \frac{5}{3}2^{n+t}$which is impossible since $B_1$ must be an integer. Therefore, $\Lambda_1 \neq 0$. 
Now, we want to apply Matveev's inequality to $\Lambda_1$ from the Theorem \ref{t:M}. To do this, we take	\begin{align*}		\alpha_1 & = 2,\\		\alpha_2 & = \alpha,\\		\alpha_3 & = \frac{5\sqrt{5}}{3}.\\		b_1 & = n,\\		b_2 & = -(k+l+m),\\		b_3 & = 1.	\end{align*}	Thus $B = k+l+m$. Further, 	\begin{equation*}		h(\alpha_1) = \log 2, \quad h(\alpha_2) = (\log \alpha)/2.	\end{equation*}		For $\alpha_3$, we use the properties of the height function (\ref{h:p1}) to deduce that		\begin{align*}		h(\alpha_3) &= h((5\sqrt{5})/3) = \log 5\sqrt{5}.	\end{align*}		From Theorem (\ref{t:M}), for $i=1,2,3$ we take $A_i$ as		\begin{align*}		A_1&\geq \max\{2\cdot h(\alpha_1), \mid \log\alpha_1\mid, 0.16\}\geq 1.38,\\		A_2&\geq \max\{2\cdot h(\alpha_2), \mid \log\alpha_2\mid, 0.16\}\geq 0.48,\\		A_3&\geq \max\{2\cdot h(\alpha_3), \mid\log\alpha_3\mid, 0.16\}\geq 4.82.	\end{align*}		Then we can take $A_1 = 1.39$, $A_2 = 0.49$, $A_3 = 4.83$. Since Theorem \ref{t:M}, we have	\begin{align*}		\log\mid \Lambda_1\mid&> -1.4\cdot 30^6\cdot 3^{4.5}\cdot 4(1+\log 2)(1+\log (k+l+m))\cdot (1.39)\cdot (0.49)\cdot (4.83)\\		&> -3.191 \times 10^{12} (1+\log (k+l+m)).	\end{align*}	Then we have from the equation (\ref{i:lam})	\begin{align*}		2k \cdot \log \alpha & \leq \log(10.73)- \log \mid \Lambda_1\mid \leq 3.192 \times 10^{12} (1+\log (k+l+m)).	\end{align*}	\begin{align*}\label{d4}		k \cdot \log \alpha & \leq 1.596 \times 10^{12} (1+\log (k+l+m)) \leq 1.596 \times 10^{12} (1+\log (3m)).	\end{align*}	Now, we rewrite the equation (\ref{e:d1}) and we fix $k$. Thus we deduce 	\begin{equation}\label{d2}		F_kF_lF_m = \frac{1}{5}F_k (\alpha^l-\beta^{l})(\alpha^m- \beta^{m}) = \frac{2^n-(-1)^n}{3},	\end{equation}	for some positive integers $k,l, m, n$.	Then we obtain that 		 \begin{equation}\label{d3}	\frac{2^n}{3} - \frac{\alpha^{l+m} \cdot F_k}{5}=\frac{(-1)^n}{3}+(-\frac{\alpha^{l}\beta^{m}}{5}-\frac{\alpha^{m}\beta^{l}}{5} + \frac{\beta^{l+m}}{5}) \cdot F_k. 	\end{equation}	We divide by $\frac{\alpha^{l+m} \cdot F_k}{5}$ both sides of the Equation (\ref{d3})		\begin{equation*}		1- \frac{2^n\alpha^{-(l+m)}\cdot 5}{3 \cdot F_k} = \frac{(-1)^{n} \cdot 5}{3 \cdot \alpha^{(l+m)}\cdot F_k}- \frac{\beta^{m}}{\alpha^{m}} - \frac{\beta^{l}}{\alpha^{l}}+\frac{\beta^{l+m}}{\alpha^{l+m}}.	\end{equation*}	Taking an absolute value of the previous equation		\begin{align*}		\lvert 	1- \frac{2^n\alpha^{-(l+m)}\cdot 5}{3 \cdot F_k} \rvert 		&\leq \lvert \frac{(-1)^{n} \cdot 5}{3 \cdot \alpha^{(l+m)}\cdot F_k}- \frac{\beta^{m}}{\alpha^{m}} - \frac{\beta^{l}}{\alpha^{l}}+\frac{\beta^{l+m}}{\alpha^{l+m}}		\rvert 		& \leq \lvert \frac{5}{3 \cdot \alpha^{(l+m)}\cdot F_k} \rvert + \frac{|\beta^{m}|}{\alpha^{m}} + \frac{|\beta^{l}|}{\alpha^{l}}+\frac{|\beta^{l+m}|}{\alpha^{l+m}}		 .	\end{align*}		Thus, we get that		\begin{equation}\label{i:l2}		\lvert 	1- \frac{2^n\alpha^{-(l+m)}\cdot 5}{3 \cdot F_k} \rvert < 3.84 \cdot {\alpha^{-2l}}.	\end{equation}		Put		\begin{equation}\label{lam:2}		\Lambda_2 = 1- \frac{2^n\alpha^{-(l+m)}\cdot 5}{3 \cdot F_k}.\end{equation}We claim that $\Lambda_2 \neq 0$. Clearly, $\Lambda_2 \in \mathbb{K}$. Suppose that $\Lambda_2 = 0$. Then\begin{equation*}2^n\alpha^{-(l+m)}\cdot 5 = 3F_k.\end{equation*}Multiplying both sides by $\alpha^{l+m}$, we obtain\begin{equation*}5\cdot 2^n = 3F_k \alpha^{l+m}.\end{equation*}Since $\alpha = (1+\sqrt{5})/2$, we have\[\alpha^{t}=\frac{A_2+B_2 \sqrt{5}}{2^{t}}\]for some integers $A_2,B_2$. Hence,\[3F_k \alpha^{l+m}=\frac{3F_kA_2}{2^{l+m}}+\frac{3F_kB_2}{2^{l+m}}\sqrt{5}.\]Multiplying both sides by $2^{l+m}$ gives\[5\cdot 2^{n+l+m}=3F_kA_2+3F_kB_2\sqrt{5}.\]The left-hand side is rational, whereas the right-hand side contains anirrational term unless $B_2=0$. Thus $B_2=0$, which implies $\alpha^{l+m}\in\mathbb{Q}$.This is impossible for $l+m>0$. Therefore, $\Lambda_2\neq 0$.		
Now, we want to apply Matveev's theorem to $\Lambda_2$ from the Theorem \ref{t:M}. To do this we take	\begin{align*}		\alpha_1 &= 2,\\		\alpha_2& = \alpha,\\		\alpha_3&=\frac{3 \cdot F_k}{5}\\		b_1 & = n,\\		b_2 & = -(l+m),\\		b_3 & = -1.	\end{align*}	Thus, $B =3m$, since $n < 3m$. Further	\begin{equation*}		h(\alpha_1) = \log 2,\quad h(\alpha_2) = (\log\alpha)/2.	\end{equation*}	For $\alpha_3$ we use the properties of the height function (\ref{h:p1}) and (\ref{b:j1}), to deduce that		\begin{align*}		h(\alpha_3) &= h(\frac{3 F_k}{5})\\		&\leq \log (3) + \log(5) + (k-l)\log \alpha\\		&\leq \log(15)+ k \log \alpha.	\end{align*} From Theorem \ref{t:M} we take 	\begin{align*}		A_1'&\geq \max\{2\cdot h(\alpha_1), \lvert \log\alpha_1\rvert, 0.16\}\geq 1.38\\		A_2' &\geq \max\{2\cdot h(\alpha_2), \lvert \log\alpha_2\rvert, 0.16\}\geq 0.48\\	\end{align*}	and since \begin{align*}	 2 \cdot h(\alpha_3) &\leq 2 \log(15) + 2 \cdot k\log \alpha\\	 &\leq 2 \log(15) + 2 \cdot 1.596 \times 10^{12} (1+\log (3m)) \\	&< 3.193 \times 10^{12} (1+ \log (3m)), \end{align*} so, $$A_3'=3.193 \times 10^{12} (1+ \log (3m)).$$Then we can take $A_1 '= 1.39$, $A_2'= 0.49$, and $A_3'= 3.193 \times 10^{12} (1+ \log (3m))$. Since Theorem \ref{t:M}, we have	\begin{align*}		\log\lvert\Lambda_2\rvert &> -1.4\cdot 30^6\cdot 3^{4.5}\cdot 4(1 + \log 2)(1 + \log (3m))(1.39)(0.49)(3.193 \times 10^{12} (1+ \log (3m)))\\		&> - 2.109 \times 10^{24}(1 + \log(3m))^2.	\end{align*}	Then, if the logarithm of both sides of the equation (\ref{i:l2}) is taken, we obtain	$$\log\lvert\Lambda_2\rvert \leq \log(3.84)- 2l \log \alpha.$$ So, $$ 2l \log \alpha \leq \log(3.84)- \log\lvert\Lambda_2\rvert \leq \log(3.84) + 2.109 \times 10^{24}(1 + \log(3m))^2 \leq 2.11 \times 10^{24}(1 + \log(3m))^2,$$ and since we have that\begin{equation*}		l \log \alpha < 1.06 \times 10^{24}(1+ \log(3m))^2.	\end{equation*}		Now, we rewrite the equation (\ref{e:d1}) and we fix $k$ and $l$. Thus we have	\begin{equation}\label{d5}		F_kF_lF_m = \frac{1}{\sqrt{5}}F_k F_l(\alpha^m- \beta^{m}) = \frac{2^n-(-1)^n}{3},	\end{equation}	for some integers $k,l, m, n$.	Then we obtain that 		 \begin{equation}\label{d6}	\frac{2^n}{3} - \frac{\alpha^{m} \cdot F_k \cdot F_l}{\sqrt{5}}=\frac{(-1)^n}{3}-\frac{\beta^{m} \cdot F_k \cdot F_l}{\sqrt{5}}	\end{equation}	We divide by $-\frac{\alpha^{m} \cdot F_k \cdot F_l}{\sqrt{5}}$ both sides of Equation (\ref{d6})		\begin{equation*} 1-		\frac{2^n\alpha^{-m}\cdot \sqrt{5}}{3 \cdot F_k F_l} =\frac{\beta^{m}}{\alpha^{m}}- \frac{(-1)^{n} \cdot \sqrt{5}}{3 \cdot \alpha^{m}\cdot F_k F_l} .	\end{equation*}	Taking an absolute value of the previous equation		\begin{align*}		\lvert 	1- \frac{2^n\alpha^{-m}\cdot \sqrt{5}}{3 \cdot F_k F_l} \rvert &\leq \lvert \frac{\beta^{m}}{\alpha^{m}}- \frac{(-1)^{n} \cdot \sqrt{5}}{3 \cdot \alpha^{m}\cdot F_k F_l} \rvert	& \leq \lvert \frac{\sqrt{5}}{3 \cdot \alpha^{m}\cdot F_k F_l} \rvert + \lvert \frac{1}{\alpha^{2m}} \rvert & \leq 1.125 \cdot \alpha^{-m}.	\end{align*}		Thus, we get that		\begin{equation}\label{d7}		\lvert 	1- \frac{2^n\alpha^{-m}\cdot \sqrt{5}}{3 \cdot F_k F_l} \rvert < 1.125 \cdot \alpha^{-m}.	\end{equation}		Put		\begin{equation}\label{lam:3}\Lambda_3 = \frac{2^n\alpha^{-m}\cdot \sqrt{5}}{3 \cdot F_k F_l}-1.\end{equation}We claim that $\Lambda_3 \neq 0$. Clearly, $\Lambda_3 \in \mathbb{K}$. Suppose that $\Lambda_3 = 0$. Then\begin{equation*}2^n\alpha^{-m}\sqrt{5} = 3F_kF_l.\end{equation*}Multiplying both sides by $\alpha^{m}$, we obtain\begin{equation*}2^n\sqrt{5} = 3F_kF_l\,\alpha^{m}.\end{equation*}Since $\alpha = (1+\sqrt{5})/2$, we have\[\alpha^{m}=\frac{A_3+B_3\sqrt{5}}{2^{m}}\]for some integers $A_3,B_3$. Substituting this into the above equality gives\[2^n\sqrt{5}=\frac{3F_kF_lA_3}{2^{m}}+\frac{3F_kF_lB_3}{2^{m}}\sqrt{5}.\]Multiplying both sides by $2^{m}$, we get\[2^{n+m}\sqrt{5}=3F_kF_lA_3+3F_kF_lB_3\sqrt{5}.\]Comparing the rational and irrational parts in $\mathbb{Q}(\sqrt{5})$, we obtain\[3F_kF_lA_3=0, \qquad 3F_kF_lB_3=2^{n+m}.\]Thus $A_3=0$ and\[B_3=\frac{2^{n+m}}{3F_kF_l}.\]This is impossible since $B_3$ must be an integer. Therefore, $\Lambda_3 \neq 0$. 		
Now, we want to apply Matveev's theorem to $\Lambda_3$ from the Theorem (\ref{t:M}). To do this we take	\begin{align*}		\alpha_1 &= 2,\\		\alpha_2& = \alpha,\\		\alpha_3&=\frac{3 F_k F_l}{\sqrt{5}}\\		b_1 & = n,\\		b_2 & = -m,\\		b_3 & = -1.	\end{align*}	Thus, $B = 3m$ since $n < 3m$. Further	\begin{align*}		h(\alpha_3) &\leq h(\frac{3 F_k F_l}{\sqrt{5}})\\		&\leq \log (3) + \log (\sqrt{5}) + \log F_k +\log F_l\\	&\leq \log (3) + \log (\sqrt{5}) + (k-1) \log \alpha + (l-1) \log \alpha\\	& < \log(3\sqrt{5}) + 2l \log \alpha.	\end{align*} So, \begin{align*}		2 \cdot h(\alpha_3)	&\leq 2 \log(3\sqrt{5}) + 2 \cdot 2l \log \alpha\\ &\leq 2 \log(3\sqrt{5}) + 4 \cdot 1.06 \times 10^{24} (1+ \log (3m))^2\\ &< 4.25\times 10^{24} (1+ \log (3m))^2.	\end{align*} From Theorem \ref{t:M}, we take 	\begin{align*}		A_1''&\geq \max\{2\cdot h(\alpha_1), \lvert \log\alpha_1\rvert, 0.16\}\geq 1.38,\\		A_2'' &\geq \max\{2\cdot h(\alpha_2), \lvert \log\alpha_2\rvert, 0.16\}\geq 0.48,\\		A_3'' & \geq \max\{2\cdot h(\alpha_3), \lvert \log\alpha_3\rvert, 0.16\}\geq 4.25 \times 10^{24} (1+ \log (3m))^2 .	\end{align*}	Then we can take $A_1 ''= 1.39$, $A_2''= 0.49$, and $A_3''= 4.25 \times 10^{24} (1+ \log (3m))^2$. Since the Theorem \ref{t:M}, we have	\begin{align*}		\log\lvert\Lambda_3\rvert &> -1.4\cdot 30^6\cdot 3^{4.5}\cdot 4(1 + \log 2)(1 + \log 3m)(1.39)(0.49)(4.23 \times 10^{24} (1+ \log (3m))^2)\\		&> - 2.794 \times 10^{36}(1+ \log(3m))^3.	\end{align*}	Then we have that from the equation (\ref{d7})	\begin{equation*}		m \log \alpha <\log 1.125 + 2.794 \times 10^{36}(1+ \log(3m))^3,	\end{equation*}	and we get \begin{equation}m \log \alpha < 2.8 \times 10^{36}(1+ \log(3m))^3.\end{equation}	With the help of SageMath, we can see that $m < 5.93 \times 10^{42}$.	We have just proved the following lemma.	\begin{lemma}\label{lem:bo}		All solution to the Diophantine equation (\ref{e:d1}) satisfy		\begin{equation*}			m < 5.93 \times 10^{42}.		\end{equation*}	\end{lemma}		The bounds given in Lemma \ref{lem:bo} are too large to carry out meaningful computing. Thus, we need to reduce them. To do this, we apply Lemma \ref{l:1} as follows.		First we return to the inequality (\ref{i:lam}) and we consider that	\begin{equation*}		\Gamma_1 := n\log 2 - (k+l+m) \log \alpha +\log(\frac{5\sqrt{5}}{3}).	\end{equation*}	We note that $e^{\Gamma_1}- 1 = \Lambda_1$. Since $\Lambda_1\neq 0$, we have that $\Gamma_1\neq 0$. If $\Gamma_1>0$, we obtain		\begin{equation*}		0<\Gamma_1\leq 1 - e^{\Gamma_1} = \lvert e^{\Gamma_1 }-1\rvert = \lvert \Lambda_1\rvert < \frac{10.73}{\alpha^{2k}}.	\end{equation*}		On the other hand, if $\Gamma_1<0$, we have that $e^{\Gamma_1}-1 =\lvert e^{\Gamma_1}-1 \rvert= \lvert \Lambda_1\rvert<\frac{1}{2}$. Thus we have	\begin{equation*}		0<\lvert\Gamma_1\rvert<e^{\lvert \Gamma_1\rvert}-1 <\frac{1}{2}.	\end{equation*}		So in the both cases we have that		\begin{equation*}		0<\lvert \Gamma_1\rvert< 2 \cdot \lvert e^{\Gamma_1}-1 \rvert <1.5 \cdot \frac{10.73}{\alpha^{2k}} = 16.095 \cdot \alpha^{-2k},	\end{equation*}		Dividing through by $\log \alpha$, we get		\begin{equation*}		0< \lvert u \tau- v+ \mu\rvert <\frac{16.095}{\log \alpha} \cdot \alpha^{-2k}.	\end{equation*}		Now, we apply Lemma \ref{l:1} with the following values	\begin{equation*}		u:=n \quad v:=(k+l+m) \quad w:= 2k \quad A :=\frac{16.095}{\log \alpha} = 33.45, \quad B :=\alpha, \quad \mu := \frac{\log(\frac{5\sqrt{5}}{3})}{\log \alpha},\quad \text{and} \quad \tau:= \frac{\log 2}{\log \alpha}.	\end{equation*}		From Lemma \ref{lem:bo}, we can write $$m < \frac{2.8 \cdot 10^{36}}{\log \alpha} (1+ \log 3m)^3 $$ and then since $1 < \log3m,$ we take $$3m < 3 (5.82 \cdot 10^{36} ) (1+\log(3m))^3 < 3\cdot 2^3 \cdot (5.82 \cdot 10^{36} )(\log (3m))^3 < 2^3 \cdot (17.46\cdot 10^{36}) (\log 3m)^3.$$ We can apply Lemma \ref{Lem} and we obtain $m < 2.94 \cdot 10^{43}.$ From inequality (\ref{1}), we have that $$n < (3m) \frac{\log \alpha}{\log 2} - 0.08 \leq 3 \cdot 2.94 \times 10^{43} \cdot \frac{\log \alpha}{\log 2} -0.08, $$ then we have $$ n < 6.13 \times 10^{43}.$$ This is an upper bound on $n$, and we get $M = 6.13 \times 10^{43}$		With the help of Mathematica, we find out that the convergent to $\tau$	\begin{equation*}		\frac{p_{92}}{q_{92}}= \frac{110504885181872237696754298142651378827960967}{76717122954194627742532516113798302944300109}	\end{equation*}	is such that $q = q_{92} > 6M$. Furthermore, it yields $\epsilon = 0.284374 >0$, and therefore 	\begin{equation*}	2k \geqslant \frac{(\log(33.45)q_{92}/\epsilon)}{\log \alpha},	\end{equation*}	then we get	 \begin{equation}\label{n-m}		k \leq 110.		\end{equation} Now, we proceed to the second case in reducing the bound on $l$. To this end, we return to Equation (\ref{i:l2}) and consider the expression	\begin{equation*}		\Gamma_2 := n\log 2 -(l+m) \log \alpha + \log(\frac{5}{3\cdot F_k}).	\end{equation*}	We note that $1-e^{\Gamma_2} = \Lambda_2$. Since $\Lambda_2\neq 0$, we have that $\Gamma_2\neq 0$. If $\Gamma_2> 0$, we obtain		\begin{equation*}		0<\Gamma_2\leq e^{\Gamma_2}-1 = \lvert1-e^{\Gamma_2}\rvert = \lvert \Lambda_2\rvert <\frac{3.84}{\alpha^{2l}}.	\end{equation*}	If $\Gamma_2<0$, we have that $1-e^{\Gamma_2} = \lvert 1- e^{\Gamma_2}\rvert = \lvert \Lambda_2\rvert < 1/2$. Then $e^{\lvert\Gamma_2\rvert}<2$. Thus,	\begin{equation*}		0<\lvert \Lambda_2\rvert < e^{\lvert\Gamma_2\rvert}-1 = e^{\lvert \Gamma_2\rvert}\lvert\Lambda_2\rvert <\frac{3.84}{\alpha^{2l}}.	\end{equation*}		So, in both cases we have that		\begin{equation*}		0< \lvert \Gamma_2\rvert < \frac{5.76}{\alpha^{2l}}.	\end{equation*}		Dividing through by $\log \alpha$ we get		\begin{equation*}		0<\lvert n\tau - m + \mu\rvert < \frac{11.97}{\alpha^{2l}}.	\end{equation*}	We now apply Lemma \ref{l:1} with the data	\begin{equation*}		A := 11.97, \quad B := \alpha, \quad \tau := \frac{\log 2}{\log \alpha}, \quad \mu := \frac{\log(\frac{5}{3\cdot F_k})}{\log \alpha}.	\end{equation*}	To do this, we rewrite inequaity (\ref{1}) for $k \leqslant 110$ and for $k < l< m$ we obtain \begin{align*}	 n &< (k+l+m) \frac{\log \alpha}{\log 2} - 0.08\\ &\leq(k+2m) \frac{\log \alpha}{\log 2} - 0.08\\ &\leq (110+2 \cdot 5.93 \times 10^{42}) \frac{\log \alpha}{\log 2} - 0.08\\ &\leq 8.24 \times 10^{42}.\end{align*} Then we take $M := 8.24 \times 10^{42}$. From Lemma \ref{lem:bo}, we have that $M$ is an upper bound on $m$.		Let $\tau =[a_0; a_1, a_2, a_3,\dots] = [1; 2, 3, 1, 2, 3, 2, 4, 2, 1, 2, 11, 2, 1, 11, 1, 1, 134, 2, 2,\dots]$ be the continued fraction expansion of $\tau$ With Mathematica we find that the convergent		\begin{equation*}		\frac{p}{q} = \frac{p_{92}}{q_{92}},	\end{equation*}	is such that $q = q_{92}= 76717122954194627742532516113798302944300109> 6M$. Furthermore, it yields $\epsilon=0.00120395 >0$,	and therefore since		\begin{equation*}		2l \geqslant \frac{\log((11.97)q/\epsilon)}{\log \alpha},	\end{equation*}we have $l \leq 114$. 		 Now, we begin to third case of the reducing bound on $m$. To do this, we return to (\ref{lam:3}) and we consider that	\begin{equation*}		\Gamma_3 := n\log 2 -m \log \alpha + \log(\frac{\sqrt{5}}{3\cdot F_k F_l})	\end{equation*}	We obtain	\begin{equation*}		0<\Gamma_3\leq e^{\Gamma_3}-1 = \lvert1-e^{\Gamma_3}\rvert = \lvert \Lambda_3\rvert <\frac{1.125}{\alpha^{m}}.	\end{equation*}		So, in both cases we have that		\begin{equation*}		0< \lvert \Gamma_3\rvert < \frac{1.6875}{\alpha^{m}}.	\end{equation*}	Dividing through by $\log \alpha$ we get		\begin{equation*}		0<\lvert n\tau - m + \mu\rvert < \frac{3.51}{\alpha^{m}}.	\end{equation*}	We now apply Lemma (\ref{l:1}) with the data	\begin{equation*}		A := 3.51, \quad B := \alpha, \quad \tau := \frac{\log 2}{\log \alpha}, \quad \mu :=\frac{\log(\frac{\sqrt{5}}{3\cdot F_k F_l})}{\log \alpha} .	\end{equation*}	To do this, for $k \leqslant 110$ and $l \leqslant 114$, we can write again inequality (\ref{1}) we take \begin{align*}	 n &< (k+l+m) \frac{\log \alpha}{\log 2} - 0.08\\ &\leq(110+ 114 +5.93 \times 10^{42}) \frac{\log \alpha}{\log 2} - 0.08\\ &\leq 2.04107 \times 10^{43},\end{align*} so we get $M:=2.04107 \times 10^{43}$ From Lemma \ref{lem:bo}, we have that $M$ is an upper bound on $m$.		Let $\tau =[a_0; a_1, a_2, a_3,\dots] = [1; 2, 3, 1, 2, 3, 2, 4, 2, 1, 2, 11, 2, 1, 11, 1, 1, 134, 2, 2,\dots]$ be the continued fraction expansion of $\tau$ With Mathematica we find that the convergent		\begin{equation*}		\frac{p}{q} = \frac{p_{93}}{q_{93}},	\end{equation*}	is such that $q = q_{93}= 464653241314711922144310025276040504720880720 > 6M$. Furthermore, it yields $\epsilon=0.0000751663>0$,	and therefore either		\begin{equation*}		m \geqslant \frac{\log((3.51)q/\epsilon)}{\log \alpha}\geqslant 236.073. 	\end{equation*}	 Thus, we have $m \leqslant 236$ so $k \leqslant l \leqslant m \leqslant 236.$ This finishes the proof.
\section*{Acknowledgements}

This work was supported by a research grant from the BK21 FOUR Program at Jeonbuk National University. The authors would like to thank the anonymous referee for his/her careful reading of the manuscript and for their valuable comments and suggestions, which have significantly improved the quality and clarity of this paper.

\section*{Author Contributions}

All authors contributed equally to all aspects of the preparation of this manuscript.

\section*{Competing Interests}

The authors declare that they have no competing interests.

\section*{Data Availability}

No data were generated or analyzed during this study.

\end{document}